\documentclass[preprint,12pt,3p]{elsarticle}

\usepackage{amssymb}
\usepackage{graphicx}
\usepackage{comment}
\usepackage{pgf,tikz}

\usetikzlibrary{snakes}
\tikzstyle{printersafe}=[snake=snake,segment amplitude=0 pt]

\usetikzlibrary{arrows}
\usetikzlibrary{shapes}
\usepackage{amsmath}
\usepackage{caption}

\newtheorem{theorem}{Theorem}

\newtheorem{proposition}{Proposition}
\newtheorem{conjecture}{Conjecture}
\newtheorem{corollary}{Corollary}

\newtheorem{lemma}{Lemma}

\newenvironment{proof}{ {\bf Proof:}} {$\Box$}

\journal{Journal of Latex templates}

\begin{document}

\begin{frontmatter}

\title{Some new results on Sylvester colorings of cubic graphs}

\author[label1]{Luca Ferrarini}
\address[label1]{Universit\'{e} Sorbonne Paris Nord, LIPN, Villetaneuse, 93430, France}


\ead{ferrarini@lipn.univ-paris13.fr}

\author[label10]{Vahan Mkrtchyan\corref{cor1}}
\address[label10]{Department of Mathematical Sciences\\
Purdue University Fort Wayne\\
Fort Wayne, IN, USA-46805}

%
\cortext[cor1]{Corresponding author}
\ead{vmkrtchy@purdue.edu}


\begin{abstract}
If $G$ and $H$ are two cubic multi-graphs, then an $H$-coloring of $G$ is a mapping $f: E(G)\rightarrow E(H)$, such that for every $v\in V(G)$ there is a vertex $x\in V(H)$, such that $f(\partial_G(v))=\partial_H(x)$. If $G$ admits an $H$-coloring then it is common to write $H\prec G$. The Petersen coloring conjecture predicts that for any bridgeless cubic graph $G$ one has $P_{10}\prec G$. Here $P_{10}$ is the Petersen graph. Let $f: E(G)\rightarrow E(H)$ be any mapping. Define: $V(f)=\{v\in V(G):\exists x\in V(H), f(\partial_G(v))=\partial_H(x)\}$. Let $S_{10}$ be the smallest cubic multi-graph that has no perfect matching. It has ten vertices. Define $S_{12}$ as the cubic graph that is obtained from $S_{10}$, by replacing its unique vertex $z$ adjacent to three bridges with a triangle. In this paper we show that (1) for every cubic multi-graph $G$ with a perfect matching, there is a mapping $f:E(G)\rightarrow E(S_{12})$, such that $|V(f)|\geq \frac{4}{5}\cdot |V(G)|$, and (2) for every cubic multi-graph $G$, there is a mapping $f:E(G)\rightarrow E(S_{10})$, such that $|V(f)|\geq \frac{5}{6}\cdot |V(G)|$. Our second result improves the $\frac{4}{5}$-bound by Hakobyan and the second author from 2018.
\end{abstract}

\begin{keyword}
Cubic Graph \sep Petersen Coloring Conjecture \sep Sylvester Coloring \sep Sylvester Coloring Conjecture
\end{keyword}

\end{frontmatter}



\section{Introduction and Notations}

In this paper, we consider finite, undirected graphs. Graphs do not contain loops, though they may contain parallel edges. In the paper, we consider graphs up to isomorphisms. This implies that the equality $G=G'$ means that $G$ and $G'$ are isomorphic.

If $G$ is a graph, then let $V(G)$ and $E(G)$ be the sets of vertices and edges of $G$, respectively. Let $\partial_{G}(v)$ be the set of edges of $G$ that are incident to the vertex $v$ of $G$. A matching of $G$ is a subset of $E(G)$ such that no two of them share a vertex. A matching of $G$ is perfect, if it contains $\frac{|V(G)|}{2}$ edges. A cut-vertex of a graph is a vertex whose removal increases the number of components of the graph. A block of $G$ is a maximal $2$-connected subgraph of $G$. An end-block is a block of $G$ containing at most one cut-vertex of $G$. A subgraph $H$ of $G$ is even, if every vertex of $H$ has even degree in $H$. 

Let $G$ be a cubic graph, and let $T$ be a triangle in $G$ such that each edge of $T$ is of multiplicity one. Such triangles will be called contractable. Let $T$ be a contractable triangle in a cubic graph and $e\in E(T)$, moreover, let $f$ be the edge of $G$ that is incident to a vertex of $T$ and is not adjacent to $e$. Then, $e$ and $f$ will be called opposite edges. We will frequently consider the graph $G/T$, which is obtained from a cubic graph $G$ by replacing a contractable triangle $T$ with a vertex. We will say that $G/T$ is obtained from $G$ by contracting the triangle $T$.

If $G$ is a cubic graph containing cut-vertices, then any end-block $B$ of $G$ is adjacent to a unique bridge $e$. We will refer to $e$ as a bridge corresponding to $B$. Furthermore, if $e=(a,b)$ and $a\in V(B)$, $b\notin V(B)$, then $b$ is called the root of $B$. We call $e$ a trivial bridge if $e$ is joined to a non-contractable triangle, otherwise it is non-trivial. An alternative view of trivial bridges is the following. Let $e$ be a bridge in a connected cubic graph $G$. Let $G_1$ and $G_2$ be the two components of $G-e$. Then $e$ is a trivial bridge, if at least one of $G_1$ or $G_2$ has three vertices.


For a positive integer $t$, a $t$-factor of $G$ is a spanning $t$-regular subgraph of $G$. Notice that the set of edges of a $1$-factor of $G$ forms a perfect matching of $G$. Moreover, if $G$ is cubic and $M$ is a $1$-factor of $G$, then the set $E(G)\backslash E(M)$ is an edge-set of a $2$-factor of $G$. This $2$-factor is said to be complementary to $M$. Conversely, if $\overline{M}$ is a $2$-factor of a cubic graph $G$, then the set $E(G)\backslash E(\overline{M})$ is an edge-set of a $1$-factor of $G$ or is a perfect matching of $G$. This $1$-factor is said to be complementary to $\overline{M}$.




Let $G$ and $H$ be two cubic graphs. An $H$-coloring of $G$ is a mapping $f:E(G)\rightarrow E(H)$, such that for each vertex $x\in V(G)$ there is a vertex $y\in V(H)$, such that $f(\partial_{G}(x)) = \partial_{H}(y)$. Let $V(f)$ denote the set of vertices of $G$ that satisfy the $H$-coloring condition under $f$, and let $\overline{V(f)} = V(G) \setminus V(f)$. If $G$ admits an $H$-coloring, then we will write $H
\prec G$. It is easy to see that if $H\prec G$ and $L\prec H$, then $L\prec G$. That is, $\prec$ induces a transitive relation on the set of cubic graphs. 



Let $P_{10}$ be the Petersen graph (see Figure \ref{PetersenGraph}). An important conjecture in the area is due to Jaeger \cite{J1988,Jaeger1975,Jaeger1979,Jaeger1985}:
\begin{conjecture}\label{P10Conjecture} ($P_{10}$-conjecture) Let $G$ be a bridgeless cubic graph. Then $P_{10}\prec G$.
\end{conjecture}

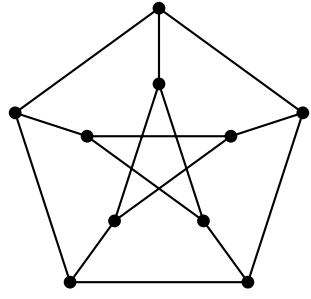
\begin{figure}[ht]
  \begin{center}
		\begin{tikzpicture}[style=thick]
\draw (18:2cm) -- (90:2cm) -- (162:2cm) -- (234:2cm) --
(306:2cm) -- cycle;
\draw (18:1cm) -- (162:1cm) -- (306:1cm) -- (90:1cm) --
(234:1cm) -- cycle;
\foreach \x in {18,90,162,234,306}{
\draw (\x:1cm) -- (\x:2cm);
\draw[fill=black] (\x:2cm) circle (2pt);
\draw[fill=black] (\x:1cm) circle (2pt);
}
\end{tikzpicture}
	\end{center}
	\caption{The graph $P_{10}$.}\label{PetersenGraph}
\end{figure}

This conjecture is very strong, as it implies many classical conjectures in the area such as:
\begin{conjecture}\label{conj:BFConjecture} (Berge-Fulkerson, 1972 \cite{Fulkerson,Seymour79}) Any bridgeless
cubic graph $G$ contains six (not necessarily distinct) perfect matchings
$F_1, \ldots , F_6$ such that any edge of $G$ belongs to exactly two of them.
\end{conjecture} 


\begin{conjecture}\label{conj:5CDC}
((5, 2)-even-subgraph-cover conjecture, \cite{Celmins1984,Preiss1981}) Any bridgeless
graph $G$ (not necessarily cubic) contains five even subgraphs such that any
edge of $G$ belongs to exactly
two of them.
\end{conjecture}
\begin{conjecture}\label{conj:7over5} (\cite{AlonTarsi1985,JT1992,CQZhangBook1,CQZhangBook2}) Any bridgeless
graph $G$ (not necessarily cubic) contains a list of cycles, such that every edge of $G$ lies on at least one of them and the sum of the lengths of these cycles is at most $\frac{7}{5}\cdot |E|$.    
\end{conjecture}

The list of six perfect matchings in Conjecture \ref{conj:BFConjecture} usually is called a Berge-Fulkerson cover of $G$. If $k(G)$ is the smallest number of perfect matchings that are needed to cover the edge-set of $G$, then observe that this conjecture implies that $k(G)\leq 5$ for any bridgeless cubic graph. This weaker statement is known as Berge conjecture, which has been shown to be equivalent to Berge-Fulkerson conjecture in \cite{Mazz11}.

Motivated by the notion of Petersen coloring, it is natural to consider $H$-colorings for other choices of the cubic graph $H$. This leads to a broader question of identifying prototype graphs that may serve as universal coloring templates for some classes of cubic graphs. 

In this direction, let $S_{4}$, $S_6$, $S_{10}$ and $S_{12}$ denote the graphs on $4,6, 10$ and $12$ vertices, respectively, as depicted on Figures \ref{S4graph}, \ref{S6graph}, \ref{SylvGraph} and \ref{fig:Sylv12graph}, respectively. 
Following \cite{Schrijver} (see page 432 of volume A), $S_{10}$ will be called Sylvester graph. It is worth noting that usually the name “Sylvester graph” is used for a particular strongly regular graph on 36 vertices, and this graph should not be confused with $S_{10}$, as it has ten vertices.

\begin{figure}[ht]
\centering
\begin{minipage}[b]{.5\textwidth}
 \begin{center}
\begin{tikzpicture}[scale = 1.2]

\node at (1.8,-0.5) {$a$};
\node at (1.6,0.1) {$a_1$};
\node at (2.5,0.1) {$a_2$};
\node at (2,0.35) {$a_3$};
\node at (2,0.8) {$a_4$};

\tikzstyle{every node}=[circle,draw,fill=black,inner sep=0pt,minimum width=4pt]

\node (n4) at (2,0) {};
\node (n5) at (1.5,0.5) {};
\node (n6) at (2.5,0.5) {};
\node (c) at (2,-1) {};

\path
(n4) edge (n5)
     edge (n6)
     edge (c)
(n5) edge (n6)
     edge[bend left] (n6);

\end{tikzpicture}

\caption{The graph $S_4$.}\label{S4graph}
\end{center}
	
\end{minipage}%
\begin{minipage}[b]{.5\textwidth}
  	\begin{center}
\begin{tikzpicture}[scale = 1.2]

\node at (1.8,-0.5) {$a$};
\node at (1.6,0.1) {$a_1$};
\node at (2.5,0.1) {$a_2$};
\node at (2,0.35) {$a_3$};
\node at (2,0.8) {$a_4$};
\node at (1.6,-1.1) {$a_5$};
\node at (2.5,-1.1) {$a_6$};
\node at (2,-1.3) {$a_7$};
\node at (2,-1.8) {$a_8$};

\tikzstyle{every node}=[circle,draw,fill=black,inner sep=0pt,minimum width=4pt]

\node (n4) at (2,0) {};
\node (n5) at (1.5,0.5) {};
\node (n6) at (2.5,0.5) {};
\node (c) at (2,-1) {};
\node (n7) at (1.5,-1.5) {};
\node (n8) at (2.5,-1.5) {};

\path
(n4) edge (n5)
     edge (n6)
     edge (c)
(n5) edge (n6)
     edge[bend left] (n6)
(n8) edge (c)
(n7) edge (n8)
     edge (c)
      edge[bend right] (n8);

\end{tikzpicture}

\caption{The graph $S_6$.}\label{S6graph}
\end{center}
\end{minipage}
\end{figure}

\begin{figure}[ht]
\centering
\begin{minipage}[b]{.5\textwidth}
  
  \begin{center}
	
		\begin{tikzpicture}[scale = 1.2]

            \node at (-5.20,-1.55) {$a$};
			\node at (-5.9,-0.85) {$a_1$};
			\node at (-5,-0.85) {$a_2$};
			\node at (-5.5,-0.65) {$a_3$};
			\node at (-5.5,-0.15) {$a_4$};
			
			\node at (-3.65,-1.55) {$b$};
			\node at (-3.9,-0.85) {$b_1$};
			\node at (-3,-0.85) {$b_2$};
			\node at (-3.5,-0.65) {$b_3$};
			\node at (-3.5,-0.15) {$b_4$};
			
			\node at (-1.85,-1.55) {$c$};
			\node at (-1.9,-0.85) {$c_1$};
			\node at (-1,-0.85) {$c_2$};
			\node at (-1.5,-0.65) {$c_3$};
			\node at (-1.5,-0.15) {$c_4$};
		
			  \tikzstyle{every node}=[circle, draw, fill=black!50,
                       inner sep=0pt, minimum width=4pt]
																		
			\node[circle,fill=black,draw] at (-5.5,-1) (n1) {};
			
			
			
			

			\node[circle,fill=black,draw] at (-6, -0.5) (n2) {};
																				
			\node[circle,fill=black,draw] at (-5,-0.5) (n3) {};
																				
			\node[circle,fill=black,draw] at (-3.5,-1) (n4) {};
																				
			\node[circle,fill=black,draw] at (-4, -0.5) (n5) {};
																				
			\node[circle,fill=black,draw] at (-3,-0.5) (n6) {};
																				
			\node[circle,fill=black,draw] at (-1.5,-1) (n7) {};
																				
			\node[circle,fill=black,draw] at (-2, -0.5) (n8) {};
																				
			\node[circle,fill=black,draw] at (-1,-0.5) (n9) {};
																				
			\node[circle,fill=black,draw] at (-3.5,-2.5) (n10) {};

			\path[every node]
			(n1) edge  (n2)

			edge  (n3)
			edge (n10) 
														 	
			(n2) edge (n3)
			edge [bend left] (n3)
															     
			(n3) 
			(n4) edge (n5)
			edge (n6)
			edge (n10)
																		    
			(n5) edge (n6)
			edge [bend left] (n6)
			(n6)
																			 
			(n7) edge (n8)
			edge (n9)
			edge (n10)
																			  
			(n8) edge (n9)
			edge [bend left] (n9)
			;
		\end{tikzpicture}
																
	\end{center}
								
	\caption{The graph $S_{10}$.}
	\label{SylvGraph}
\end{minipage}
\begin{minipage}[b]{.5\textwidth}
\centering
  \begin{center}

		\begin{tikzpicture}[scale = 1.2]
			
			\node at (-5.65,-1.55) {$a$};
			\node at (-5.9,-0.85) {$a_1$};
			\node at (-5,-0.85) {$a_2$};
			\node at (-5.5,-0.65) {$a_3$};
			\node at (-5.5,-0.15) {$a_4$};
			
			\node at (-3.65,-1.55) {$b$};
			\node at (-3.9,-0.85) {$b_1$};
			\node at (-3,-0.85) {$b_2$};
			\node at (-3.5,-0.65) {$b_3$};
			\node at (-3.5,-0.15) {$b_4$};
			
			\node at (-1.65,-1.55) {$c$};
			\node at (-1.9,-0.85) {$c_1$};
			\node at (-1,-0.85) {$c_2$};
			\node at (-1.5,-0.65) {$c_3$};
			\node at (-1.5,-0.15) {$c_4$};
			
			\node at (-4.65,-1.75) {$z$};
			\node at (-2.65,-1.75) {$x$};
			\node at (-3.6,-2.4) {$y$};
			
			\tikzstyle{every node}=[circle, draw, fill=black!50,
                        inner sep=0pt, minimum width=4pt]
																								
			\node[circle,fill=black,draw] at (-5.5,-1) (n1) {};
																								
			\node[circle,fill=black,draw] at (-6, -0.5) (n2) {};
																								
			\node[circle,fill=black,draw] at (-5,-0.5) (n3) {};
																								
			\node[circle,fill=black,draw] at (-3.5,-1) (n4) {};
																								
			\node[circle,fill=black,draw] at (-4, -0.5) (n5) {};
																								
			\node[circle,fill=black,draw] at (-3,-0.5) (n6) {};
																								
			\node[circle,fill=black,draw] at (-1.5,-1) (n7) {};
																								
			\node[circle,fill=black,draw] at (-2, -0.5) (n8) {};
																								
			\node[circle,fill=black,draw] at (-1,-0.5) (n9) {};
																								
			\node[circle,fill=black,draw] at (-5.5,-2) (n10) {};
																								
			\node[circle,fill=black,draw] at (-3.5,-2) (n11) {};
																								 
			\node[circle,fill=black,draw] at (-1.5,-2) (n12) {};
																								
			\path[every node]
			(n1) edge  (n2)
																								    
			edge  (n3)
			edge (n10)
																								   	
			(n2) edge (n3)
			edge [bend left] (n3)
																								       
			(n3) 
			(n4) edge (n5)
			edge (n6)
			edge (n11)
																								    
			(n5) edge (n6)
			edge [bend left] (n6)
			(n6)
																								   
			(n7) edge (n8)
			edge (n9)
			edge (n12)
																								    
			(n8) edge (n9)
			edge [bend left] (n9)
																								   
			(n10) edge (n11)
			edge (n12)
																								   
			(n10) edge [bend right] (n12)
																								  
			;
		\end{tikzpicture}
																
	\end{center}
	
	\caption{The graph $S_{12}$.}\label{fig:Sylv12graph}
\end{minipage}
\end{figure}

In \cite{AuJC2018,HM2019,Mkrt2013} it is conjectured that any cubic graph $G$ admits an $S_{10}$-coloring, and any cubic graph $G$ with a perfect matching admits an $S_{12}$-coloring. These two conjectures have been disproved recently in \cite{WolfAMC2026}, where the author has given an example of a cubic graph with a perfect matching that does not admit an $S_{10}$-coloring (hence it cannot admit an $S_{12}$-coloring). On the positive side, in \cite{KMZ22} it is shown that any bridgeless cubic graph admits an $S_6$-coloring (see Theorem \ref{thm:KMZJCTB} below). Note that this is the same as admitting an $S_4$-coloring.







In this paper, we establish the following two results. First, we show that for any cubic graph $G$ with a perfect matching there is a mapping $f:E(G)\rightarrow E(S_{12})$ such that $|V(f)| \geq \frac{4}{5}\cdot |V(G)|$. Then we show that for any cubic graph $G$ there is a mapping $f: E(G) \longrightarrow E(S_{10})$ such that $|V(f)| \geq \frac{5}{6}\cdot |V(G)|$. This improves the $\frac{4}{5}\cdot |V(G)|$-bound proved in \cite{AuJC2018}. Non-defined concepts can be found in \cite{west:1996}.



\section{Auxiliary Results}\label{sec:aux}

In this section we introduce some auxiliary results which turn out to be useful for deriving our main results. Let $G$ be a graph of maximum degree at most three, and let $c$ be a proper coloring of a subset of $E(G)$ using three colors $\{1,2,3\}$. An edge in $G$ is defined as \textit{uncolored} if it has not received  a color under $c$. Assume that $c$ is chosen such that the number of uncolored edges of $G$ is the smallest. Take an arbitrary uncolored edge $e=uv$. Since $e$ is uncolored, the edges adjacent to $e$ receive colors 1,2 and 3 in such a way that no two edges incident to the same vertex share a color; otherwise, we can color the edge $e$ with the remaining available color. Thus, there exist two colors $\alpha,\beta \in \{1,2,3\}$, such that $\alpha$ is absent at $u$ and $\beta$ is absent at $v$. Consider a maximal path $P_{e}$ starting from $\alpha$ and ending in $\beta$. By the minimality of $c$, $P_{e}$ must conclude at $v$, and the length of this path must be even; otherwise, we can switch the colors along $P_e$ to allow $e$ to be colored, yielding a contradiction. Hence, let $C_{e}$ be the odd cycle corresponding to the uncolored edge $e$ formed by $P_{e}$ together with the edge $e$. This leads to consider the following lemma:

		\begin{lemma}\label{lemma uncolored}(\cite{Fouquet,Steffen1,Steffen2})
			If $G$ is a graph of maximum degree at most three, and $ e,e'$ $(e \neq e') $ are two uncolored edges of $ G $, then $ V(C_{e}) \cap V(C_{e'})=\emptyset $.
		\end{lemma}

Next, we will need the following recent result.
\begin{theorem}
    \label{thm:KMZJCTB} (\cite{KMZ22}) Every bridgeless cubic graph admits an $S_4$-coloring.
\end{theorem}	

\begin{corollary}
    \label{cor:kBridgeColoring} Let $G$ be a cubic graph in which all bridges are trivial. Let $\theta$ be the number of bridges of $G$. Then there is a mapping $f:E(G)\rightarrow E(S_{10})$, such that
    \[|\overline{V(f)}|\leq \theta,\]
    and for every vertex $z\in \overline{V(f)}$, there are $e,e'\in \partial_G(z)$, such that $f(e)=f(e')$.
\end{corollary}
\begin{proof} Consider a graph $G'$ obtained from $G$ by removing the vertices of all $\theta$ non-contractable triangles. Note that $G'$ is a graph of maximum degree at most three, in which we have $\theta$ vertices of degree two. A path $P$ of degree-two vertices in $G'$ is a simple path joining vertices $v_1$ and $v_2$, such that $v_1$ and $v_2$ are of degree three and all internal vertices of $P$ are of degree two in $G'$. Consider a cubic graph $W'$ obtained from $G'$ by removing all the internal vertices of every such a path $P$ and joining their respective end-points $v_1$ and $v_2$ with an edge $v_1v_2$. Note that $W'$ is a bridgeless cubic graph, and hence it has an $S_4$-coloring (see Theorem \ref{thm:KMZJCTB}). Consider a coloring $f$ of $G$ obtained as follows: for every path $P$ of degree two vertices in $G'$, let $v_1$ and $v_2$ be the two end-vertices of $P$. Color all edges of $P$ with the color of $v_1v_2$ in $W'$, then color the trivial bridges with color $a$, and then extend it to the edges of the non-contractable triangle by using colors $a_1, a_2, a_3$ and $a_4$.

    Note that for this coloring we have:
    \[|\overline{V(f)}|\leq \theta,\]
    and for every vertex $z\in \overline{V(f)}$, there are $e,e'\in \partial_G(z)$, such that $f(e)=f(e')$.
\end{proof}

\section{Main Results}

In this section, we obtain the main results of the paper. In \cite{AuJC2018}, it is shown that:

\begin{theorem}
    \label{thm:|V|5Bound} (\cite{AuJC2018}) Let $G$ be a cubic graph. Then, there is a mapping $f:E(G)\rightarrow E(S_{10})$, such that 
    \[|V(f)|\geq \frac{4}{5}\cdot |V(G)|.\]
\end{theorem}

The main idea of the proof of this theorem given in \cite{AuJC2018} relies on the structural properties of maximum $3$-edge-colorable subgraphs in graphs of maximum degree three and in particular, on Lemma \ref{lemma uncolored}. One may wonder whether something similar can be proved for cubic-multigraphs with perfect matchings, where instead of $S_{10}$, we consider edges of $S_{12}$ as colors. Our first result answers this question.

\begin{figure}
	\begin{center}
         \begin{tikzpicture}[scale = 0.70]

              \begin{scope}
              \clip (-4,0) rectangle (4,4);
              \draw (0,0) circle(4);
              \draw (-4,0) -- (4,0);
              \end{scope}

              \coordinate (c) at (0,4);
              \coordinate (a) at (-4,0);
              \coordinate (b) at (4,0);
              \coordinate (a1) at (2.8284,2.8284);
              \coordinate (b1) at (-2.8284,2.8284);
              \coordinate (c1) at (0,6.5);
              \coordinate (b2) at (-4,5);
              \coordinate (a2) at (4,5);
              \draw[fill=black] (c) circle [radius = 0.15 cm];
              \draw[fill=black] (a) circle [radius = 0.15 cm];
              \draw[fill=black] (b) circle [radius = 0.15 cm];
              \draw[fill=black] (a1) circle [radius = 0.15 cm];
              \draw[fill=black] (b1) circle [radius = 0.15 cm];
              \node at (-4,1.5) {$\alpha$};
              \node at (4,1.5)  {$\beta$};
              \node at (1.8,3.8)  {$\alpha$};
              \node at (-1.8,3.8)  {$\beta$};
              \node at (0,-0.2)  {$e$};
              \node at (3.7,3.7) {$\gamma$};
              \node at (-3.7,3.7) {$\gamma$};
              \node at (-0.3,5.2) {$\gamma$};
              \coordinate (z) at (-4.8,-2);
              \coordinate (w) at (4.8,-2);
              \draw (a) -- (z);
              \draw (b) -- (w);
              \draw (b1) -- (b2);
              \draw (c) -- (c1);
              \draw (a1) -- (a2);
              \node at (-4.6,-1) {$\gamma$};
              \node at (4.6,-1)  {$\gamma$};
         \end{tikzpicture}
     \end{center}
\caption{An odd cycle corresponding to an \textit{uncolored edge} $e$}
\end{figure}
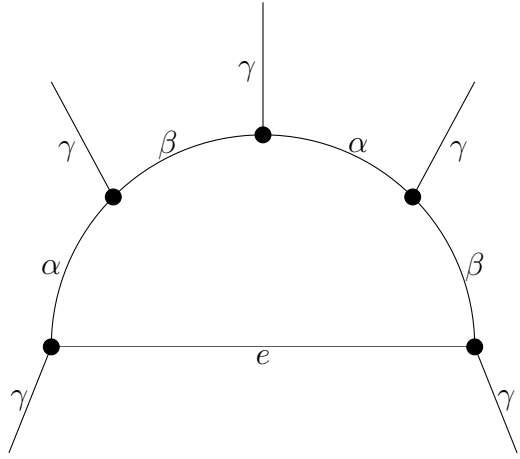

	\begin{theorem}
			Let G be a cubic graph with a perfect matching. Then, there exists a mapping  $f:E(G)\longrightarrow E(S_{12})$, such that:
			\begin{center}
				$|V(f)| \geq \dfrac{4}{5}\cdot|V(G)|$
			\end{center} 
			and for any $u \in \overline{V(f)}$ there are two edges $e,e' \in \partial_G(u)$ such that $f(e)=f(e')$.
		\end{theorem}

		\begin{proof}
		\begin{flushleft}
		
			We prove the theorem by induction on the number of vertices. If $|V(G)|=2$, then the resulting graph is $K_{2}^3$, the unique graph with two vertices and three parallel edges. Color the edges of this graph with the edges incident to an arbitrary vertex $u \in S_{12}$, and hence the base case is satisfied. 

            \smallskip
            
            Now, assume that the statement holds for all cubic graphs with $|V(G)|< n $, and consider a cubic graph $G$ with $|V(G)|=n\geq 4$ vertices. Choose an arbitrary contractable triangle $\textit{T}$ and consider the cubic graph $W:=G/T$. Let $v_{T}$ be the vertex of $W$ obtained by contracting $ T $. 
			We will consider two cases. 
			
		\end{flushleft}
			\begin{flushleft}
				$Case$ 1: $G/T$ has a perfect matching.
			\end{flushleft}  
			Since $W$ has $n-2$ vertices, by the inductive hypothesis there is a mapping $h: E(W) \longrightarrow E(S_{12})$, such that:
			\begin{center}
				$|V(h)| \geq \dfrac{4}{5}\cdot |V(W)|$,
			\end{center} 
			\begin{flushleft}
				and for any $u \in V(W)\backslash V(h)$ there are two edges $e,e' \in \partial_W(u)$ such that $h(e)=h(e')$. We consider two subcases.
			\end{flushleft}
			$subcase$ 1.1: $v_{T} \in V(h)$.\\
			
			\begin{flushleft}
				Then, there is a vertex $s \in V(S_{12})$, such that $\partial_{S_{12}} (s)= \{ \alpha, \beta, \gamma \}$ and $h(\partial_W(v_{T}))= \partial_{S_{12}} (s)$. Let $\alpha$ be the edge of $S_{12}$ lying in a perfect matching of $S_{12}$. Consider the mapping $f:E(G)\longrightarrow E(S_{12})$ obtained from $h$ as follows: color every edge $e$ of $T$ with a color from $\{\alpha,\beta, \gamma\}$ that is used to color the opposite edge of $e$ in $h$. The corresponding coloring is shown in the Figure \ref{fig:Thasperfectmatching}. Observe that since $W$ has a perfect matching by hypothesis, the only possible configuration for this case is shown in the Figure \ref{fig:Thasperfectmatching}. We obtain:\\
			\end{flushleft}
			\begin{center}
				$ |V(f)|=|V(h)|+2 $, and $|V(G)|=|V(W)|+2$,
			\end{center}
			thus
			\begin{center}
				$ \dfrac{|V(f)|}{|V(G)|}=\dfrac{|V(h)|+2}{|V(W)|+2} \geq \dfrac{|V(h)|}{|V(W)|} \geq \dfrac{4}{5}$,
			\end{center}
			and we conclude that
			\begin{center}
				$|V(f)| \geq \dfrac{4}{5}\cdot|V(G)|$.
			\end{center}  
			\begin{flushleft}
			
	\begin{figure}[!htbp]
     \begin{center}
	   \begin{tikzpicture}[scale=0.54]          
            \draw[fill=black] (0,0) circle [radius=0.15cm] ;
            \draw[fill=black] (5,0) circle [radius=0.15cm];
            \draw[fill=black] (2.5,4) circle [radius=0.15cm] ;
            \draw [blue] (0,0) -- (5,0); 
            \draw (5,0)-- (2.5,4) --(0,0);
            \draw [blue] (2.5,4)--(2.5,7); 
            \draw  (0,0) -- (-1.5,-2);
            \draw  (5,0) -- (6.4,-1.9); 
            \node at (2.25,5.5) {$\alpha$};
            \node at (-1.1,-1) {$\beta$};
            \node at (6.2,-1){$\gamma$};
            \node at (2.5,-0.3) {$\alpha$};
            \node at (0.7,2) {$\gamma$};
            \node at (4.3,2) {$\beta$};
            \node at (1.9,4) {$v_{T}$};
            
            \draw[fill=black] (-8,2.5) circle [radius=0.15cm];
            \node at (-8.5,2.5) {$v_{T}$};
            \draw[blue] (-8,2.5)--(-8,5.5);
            \node at (-7.6,3.7) {$\alpha$};
            \draw (-8,2.5)--(-10,-0.5);
            \node at (-9.2,1.2) {$\beta$};
            \draw (-8,2.5)--(-6,-0.5);
            \node at (-6.8,1.2) {$\gamma$};
        \end{tikzpicture}
        
        \caption{Coloring of a contractable triangle: the blue edges lie in a perfect matching.}\label{fig:Thasperfectmatching}
        \end{center}
        \end{figure}
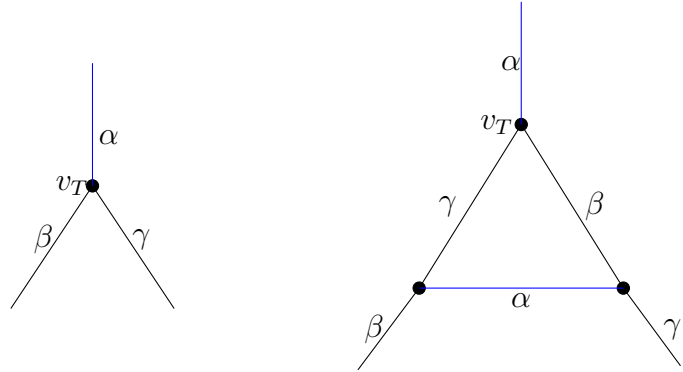	
				$ subcase $ 1.2: $v_{T} \notin V(h)$.
			\end{flushleft}
            
            
			Since $W$ has a perfect matching and $|V(W)|<n$, there are two edges $e,e' \in \partial_{S_{12}}(v_{T})$, such that $ h(e)=h(e')$. Consider three adjacent edges in $S_{12}$ such that $h(e)=\eta$ and let $ \sigma $ and $ \omega $ be the other two incident to the same vertex of $\eta$. 
            
            Let $f:E(G)\longrightarrow E(S_{12})$ be a mapping obtained from $ h $ as follows: color the two edges of $T$ that are opposite to the edges of color $\eta$ in $h$, with color $\sigma$, and the remaining edge that is opposite to an edge of color $\omega$ with color $\omega$.        
             See the coloring on Figure \ref{fig:subcase}. As in the previous case, we have that:
             \begin{center}
				$ |V(f)|=|V(h)|+2 $, and $|V(G)|=|V(W)|+2$,
			\end{center}
			and therefore
			\begin{center}
				$ \dfrac{|V(f)|}{|V(G)|}=\dfrac{|V(h)|+2}{|V(W)|+2} \geq \dfrac{|V(h)|}{|V(W)|} \geq \dfrac{4}{5}$.
			\end{center}
			We conclude that
			\begin{center}
				$|V(f)| \geq \dfrac{4}{5}\cdot|V(G)|$.
			\end{center}
			
			\begin{figure}
		        \begin{center}
		        
	        \begin{tikzpicture}[scale=0.54]          
            \draw[fill=black] (0,0) circle [radius=0.15cm] ;
            \draw[fill=black] (5,0) circle [radius=0.15cm];
            \draw[fill=black] (2.5,4) circle [radius=0.15cm] ;
            \draw [blue] (0,0) -- (5,0); 
            \draw (5,0)-- (2.5,4) --(0,0);
            \draw [blue] (2.5,4)--(2.5,7); 
            \draw  (0,0) -- (-1.5,-2);
            \draw  (5,0) -- (6.4,-1.9); 
            \node at (2.25,5.5) {$\omega$};
            \node at (-1.1,-1) {$\eta$};
            \node at (6.2,-1){$\eta$};
            \node at (2.5,-0.3) {$\omega$};
            \node at (0.7,2) {$\sigma$};
            \node at (4.3,2) {$\sigma$};
            \node at (1.9,4) {$v_{T}$};
            
            \draw[fill=black] (-8,2.5) circle [radius=0.15cm];
            \node at (-8.5,2.5) {$v_{T}$};
            \draw[blue] (-8,2.5)--(-8,5.5);
            \node at (-7.6,3.7) {$\omega$};
            \draw (-8,2.5)--(-10,-0.5);
            \node at (-9.2,1.2) {$\eta$};
            \draw (-8,2.5)--(-6,-0.5);
            \node at (-6.8,1.2) {$\eta$};
        \end{tikzpicture}
        
        \caption{The coloring of a contractable triangle such that $v_{T} \notin V(h)$.}\label{fig:subcase}
        \end{center}
			\end{figure}
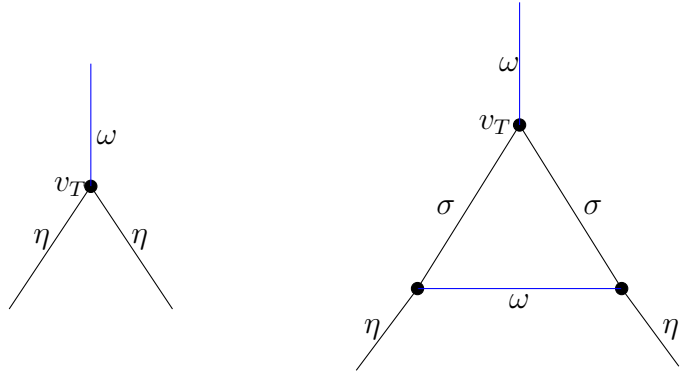
		
			$Case$ 2: $G/T$ has no perfect matching for any contractable triangle $T$. \\ \\
			Since $G$ has a perfect matching, the edges incident to the vertices of a contractable triangle must lie in every perfect matching of $G$. We present a method for coloring the edges of $G$. 
			
			Consider the graph $G'$ obtained from $G$ by removing the vertices of all non-contractable triangles. In $G'$, we have only contractable triangles such that $G/T$ has no perfect matching for any contractable triangle $T$. In particular, this means that these triangles are vertex-disjoint in $G'$.
            Let $W$ be the graph obtained from $G'$ by contracting all the triangles in $G'$. Observe that in $W$, the degree of each vertex is either $2$ or $3$. Now, consider a coloring of edges of $W$ with the bridges $a$,$b$ and $c$ of $S_{12}$ (as labeled on Figure \ref{fig:Sylv12graph}) such that the number of uncolored edges is smallest. Let $e$ be an uncolored edge. We claim that the corresponding cycle $C_{e}$ contains a vertex $v_{e}$ which does not correspond to a contracted triangle in $W$. Suppose by contradiction that all vertices on $C_{e}$ correspond to contracted triangles. By replacing each vertex of \(C_{e}\) with its original triangle, we obtain a cycle of triangles as shown in Figure \ref{fig:cycletriangles}. This configuration contradicts our hypothesis, as it is possible to shift the colors along the cycle such that the contraction of the triangles yields a perfect matching, which brings us to a configuration from $ Case $ 1. We conclude that $C_{e}$ contains at least one vertex $v_{e}$ that does not correspond to a contractable triangle. 
            Note that, since the vertices in $W$ have degree at most 3,  Lemma \ref{lemma uncolored} implies that the cycles corresponding to uncolored edges are vertex-disjoint, meaning $C_e$ contains at most one such vertex. We now extend this coloring to the edges of $G'$. In $G'$, the edges of the contractable triangles belonging to a perfect matching receive the color $a$,$b$, and $c$ while their adjacent edges receive colors $x$, $y$ and $z$ as shown in the Figure \ref{fig:secondcase} (see also Figure \ref{fig:Sylv12graph}). This implies that the cycles corresponding to uncolored edges yield a proper coloring in $G'$ for vertices belonging to contractable triangles, leaving at most one vertex that does not satisfy the condition. Now, let $T'$ be a non-contractable triangle of $G$. Extend this coloring to that of $G$. Color the edges of $T'$ to the corresponding edges of the end-block of $S_{12}$. Clearly, the bridges lie in a perfect matching, and so they receive a color $d \in \{a,b,c\}$. Thus, the edges of the component of $T'$ that does not contain the root are colored with the corresponding ones in $S_{12}$ in such a way to preserve the adjacency. Hence, \[|\overline{V(f)}| \leq \frac{|V(G')|-3\tau}{5}\leq \frac{|V(G')|}{5}\leq \frac{|V(G)|}{5},\]
            where $\tau$ is the number of contractable triangles. We conclude that
			\begin{center}
				$|V(f)| \geq \dfrac{4}{5}\cdot|V(G)|$.
			\end{center}  
            Note that by the description of this method, we have that for all vertices $v\in \overline{V(f)}$, there are edges $e,e'\in \partial_G(v)$, such that $f(e)=f(e')$.
            
\begin{figure}[!htbp]
 \centering

  \begin{center}
	  \begin{tikzpicture}[scale=0.54]

       \draw[fill=black] (0,0) circle [radius=0.15cm] ;
       \draw[fill=black] (5,0) circle [radius=0.15cm];
       \draw[fill=black] (2.5,4) circle [radius=0.15cm] ;
       \draw (0,0) -- (5,0) -- (2.5,4) --(0,0);
       \draw [blue] (2.5,4)--(2.5,7); 
       \draw [blue]  (0,0) -- (-1.5,-2);
       \draw [blue] (5,0) -- (6.4,-1.9); 
       \node at (2.25,5.5) {$a$};
       \node at (-1.1,-1) {$b$};
       \node at (6.2,-1){$c$};
       \node at (2.5,-0.3) {$x$};
       \node at (0.7,2) {$z$};
       \node at (4.3,2) {$y$};
  
\end{tikzpicture}
   \caption{The coloring with the edges of $S_{12}$ of a contractable triangle $T$ such that $G/T$ has no perfect matching}\label{fig:secondcase}
\end{center}
\end{figure}
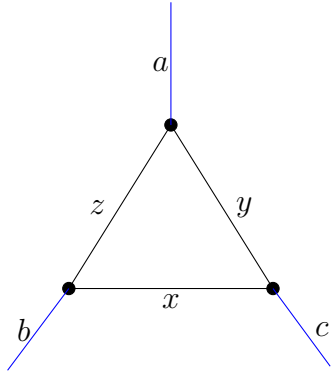

\begin{figure}[!htbp]
\begin{center}
\begin{tikzpicture}[scale=0.54]

	\coordinate (a2) at (14.75,-2.6+2.6);
	\coordinate (a22) at (15.75,-3.6+2.6);
	\coordinate (b2) at (13,-2+2.6);
	\coordinate (c2) at (14,-1+2.6);
	
	
	\draw[fill=black] (a2) circle [radius=0.15cm] ;
	
	\draw[fill=black] (b2) circle [radius=0.15cm] ;

	\draw[fill=black] (c2) circle [radius=0.15cm] ;
	
	\draw (a2)--(b2)--(c2)--(a2);
	\draw[blue] (a2) -- (a22);

	
	\coordinate (a3) at (14.65,6.75);
	\coordinate (a33) at (15.65,7.75);
	\coordinate (b3) at (14,5);
	\coordinate (c3) at (13,6);

	
	\draw[fill=black] (a3) circle [radius=0.15cm] ;
	
	\draw[fill=black] (b3) circle [radius=0.15cm] ;
	
	\draw[fill=black] (c3) circle [radius=0.15cm] ;

	\draw (a3)--(b3)--(c3)--(a3);
	\draw[blue] (a3) -- (a33);

	
	\coordinate (a4) at (5.33,6.66);
	\coordinate (a44) at (4.33,7.66);
	\coordinate (b4) at (7,6);
	\coordinate (c4) at (6,5);
	
	
	\draw[fill=black] (a4) circle [radius=0.15cm] ;
	
	\draw[fill=black] (b4) circle [radius=0.15cm] ;
	
	\draw[fill=black] (c4) circle [radius=0.15cm] ;
	
	\draw (a4)--(b4)--(c4)--(a4);
	\draw[blue] (a4) -- (a44);

    	\coordinate (a1) at (5.23,-2.7+2.6);
    	\coordinate (a11) at (4.23,-3.7+2.6);
    	\coordinate (b1) at (6,-1+2.6);
    	\coordinate (c1) at (7,-2+2.6);
    	
    	\draw[fill=black] (a1) circle [radius=0.15cm] ;
    	
    	\draw[fill=black] (b1) circle [radius=0.15cm] ;
    	
    	\draw[fill=black] (c1) circle [radius=0.15cm] ;
    	
    	\draw (a1)--(b1)--(c1)--(a1);
    	\draw[blue] (a1) -- (a11);
    	
       
       	
	   \coordinate (a5) at (9.9,9.8);
	   \coordinate (a55) at (9.9,11.2);
	   \coordinate (b5) at (10.7,8);
	   \coordinate (c5) at (9.2,8);
    	
       \draw[fill=black] (a5) circle [radius=0.15cm] ;
    	
        \draw[fill=black] (b5) circle [radius=0.15cm] ;
    	
        \draw[fill=black] (c5) circle [radius=0.15cm] ;
    	
    	\draw (a5)--(b5)--(c5)--(a5);
    	\draw[blue] (a5) -- (a55);
    	
    	\draw [blue] (b4) --(c5);
    	\draw [blue] (b5) -- (c3);
    	\draw [blue] (c4) -- (b1);
    	\draw [blue] (c1) -- (b2);
    	\draw [blue] (c2) --(b3);

\end{tikzpicture}
 \caption{this situation is not admissible since $G/T$ has no perfect matching}
\label{fig:cycletriangles}
\end{center}
\end{figure}
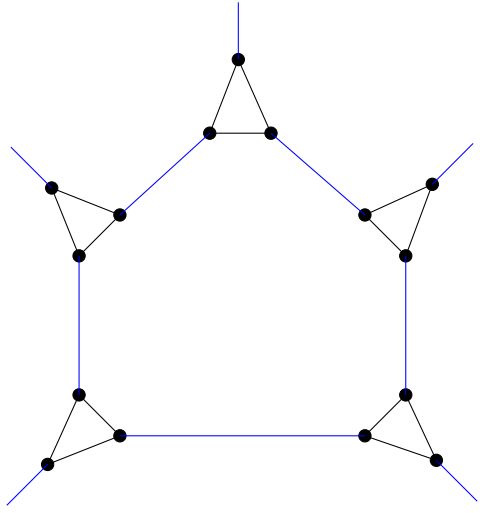
			
		\end{proof}

Theorem \ref{thm:|V|5Bound} proved in \cite{AuJC2018} shows that for every cubic graph there is a mapping $f:E(G)\rightarrow E(S_{10})$, such that $|V(f)|\geq \frac{4}{5}\cdot |V(G)|$. One may wonder whether this bound can be improved. This question is right to ask because of the recent proof of Theorem \ref{thm:KMZJCTB}. Our next result answers it.

\begin{theorem}
    \label{thm:|V|6Bound} Let $G$ be a cubic graph. Then, there is a mapping $f:E(G)\rightarrow E(S_{10})$, such that 
    \[|V(f)|\geq \frac{5}{6}\cdot |V(G)|,\]
    and for every $v\in \overline{V(f)}$ there are $e,e\in \partial_G(v)$ with $f(e)=f(e')$.
\end{theorem}

\begin{proof} We prove the theorem by induction on the number of vertices. If $|V(G)|= 2$, we can take an arbitrary vertex $w$ of $S_{10}$, and color the three edges of $G$ with three edges incident to $w$. It is trivial to see that this coloring satisfies the condition of the theorem. Now, assume that the statement of the theorem holds for all cubic graphs with $|V(G)|<n$. Let us consider an arbitrary cubic graph $G$ containing $n\geq 4$ vertices. Clearly, we can assume that $G$ is connected.

Let us show that all triangles in $G$ are non-contractable. Suppose, $G$ contains a contractable triangle $T$. Consider the cubic graph $H=G/T$, and let $v_T$ be the vertex of $H$ obtained by contracting $T$. Since $H$ contains $n-2$ vertices, we have that there is a mapping $g:E(H)\rightarrow E(S_{10})$, such that
\begin{equation*}
|V(g)|\geq \frac{5}{6}\cdot |V(H)|,
\end{equation*}and for any $v \in \overline{V(g)}$ there are two edges $e, e' \in \partial_H(v)$, such that $g(e)=g(e')$. We consider two cases.

\medskip

Case A: $v_T\in V(g)$.

\medskip

There is a vertex $s\in V(S_{10})$, such that $g(\partial_H(v))=\partial_{S_{10}}(s)$. Let $\partial_{S_{10}}(s)=\{\alpha, \beta, \gamma \}$. Consider a mapping $f:E(G)\rightarrow E(S_{10})$, obtained from $g$ as follows: color the edges of $T$ with a color from $\{\alpha, \beta, \gamma \}$, such that its end-vertices are not incident to an edge with that color. Observe that
\begin{equation*}
|V(f)|=|V(g)|+2, \text{ and } |V(G)|=|V(H)|+2,
\end{equation*}hence
\begin{equation*}
\frac{|V(f)|}{|V(G)|}=\frac{|V(g)|+2}{|V(H)|+2}\geq \frac{|V(g)|}{|V(H)|}\geq \frac{5}{6},
\end{equation*}and therefore,
\begin{equation*}
|V(f)|\geq \frac{5}{6}\cdot |V(G)|.
\end{equation*}

\medskip

Case B: $v_T\notin V(g)$.

\medskip

There are two edges $e, e' \in \partial_H(v_T)$, such that $g(e)=g(e')$. Let $x=g(e)$, and let $y$ and $z$ be two edges of $S_{10}$ that are incident to the same end-vertex of $x$ in $S$. Consider a mapping $f:E(G)\rightarrow E(S_{10})$, obtained from $g$ as follows: color the edges of $T$ that are opposite to the edges with color $x$ by $y$, and color the remaining third edge of $T$ with $z$. Observe that
\begin{equation*}
|V(f)|=|V(g)|+2, \text{ and } |V(G)|=|V(H)|+2,
\end{equation*}hence
\begin{equation*}
\frac{|V(f)|}{|V(G)|}=\frac{|V(g)|+2}{|V(H)|+2}\geq \frac{|V(g)|}{|V(H)|}\geq \frac{5}{6},
\end{equation*}and therefore,
\begin{equation*}
|V(f)|\geq \frac{5}{6}\cdot |V(G)|.
\end{equation*} Moreover, for each vertex $w \notin V(f)$, there are two edges $h, h' \in \partial_G(w)$, such that $f(h)=f(h')$.

\medskip

Thus, we can assume that all triangles of $G$ are non-contractable. Let $t$ be the number of trivial bridges of $G$, and $k$ be the number of non-trivial bridges of $G$. Let $G'$ be the subgraph of $G$ obtained from $G$ by removing the vertices of all $t$ non-contractable triangles of $G$. Notice that $G'$ is triangle free with maximum degree at most three.  
Now, consider the following cases.

\medskip

\textit{Case 1}: $|V(G')|\leq \frac{5}{6}\cdot |V(G)|$.

\medskip

Consider a coloring $f$ by using the edges $a,b,$ and $c$ of $S_{10}$ (see Figure \ref{SylvGraph}) and such that the number of uncolored edges is smallest. We extend this coloring to that of $G$ as follows. Let $T'$ be a non-contractable triangle of $G$. Color the edges of $T'$ with the corresponding edges of the end-block of $S_{10}$. Clearly, the bridges receive a color $d \in \{a,b,c\}$. Thus, the edges of the component of $T'$ that does not contain the root are colored with the corresponding ones in $S_{10}$ in such a way to preserve the adjacency. Since $G'$ is triangle free, every cycle associated with an uncolored edge has length at least $5$. Moreover, by Lemma~\ref{lemma uncolored}, these cycles are vertex-disjoint. Hence:
\[|\overline{V(f)}|\leq \frac{|V(G')|}{5}\leq \frac{|V(G)|}{6},\]
and therefore,
\begin{equation*}
|V(f)|\geq \frac{5}{6}\cdot |V(G)|,
\end{equation*}
and the coloring we get satisfies the property: for every $v\in \overline{V(f)}$ there are $e,e\in \partial_G(v)$ with $f(e)=f(e')$.

\medskip

\textit{Case 2}: $|V(G')|\geq \frac{5}{6}\cdot |V(G)|$.

\medskip

Since 
\[|V(G)|=|V(G')|+3\cdot t\geq \frac{5}{6}\cdot |V(G)|+3\cdot t,\]
we have
\[t\leq \frac{|V(G)|}{18}.\]

We will partition this case into two subcases.

\medskip

Case 2.1: $k\leq \frac{|V(G)|}{18}$. Note that by Corollary \ref{cor:kBridgeColoring} we have that there is a mapping $f:E(G)\rightarrow E(S_{10})$ such that 
\[|\overline{V(f)}|\leq t+2\cdot k,\]
and the coloring we get satisfies: for every $v\in \overline{V(f)}$ there are $e,e\in \partial_G(v)$ with $f(e)=f(e')$.
Hence
\[|\overline{V(f)}|\leq t+2\cdot k\leq \frac{|V(G)|}{18}+2\cdot \frac{|V(G)|}{18}=\frac{|V(G)|}{6}.\]
Therefore,
\begin{equation*}
|V(f)|\geq \frac{5}{6}\cdot |V(G)|.
\end{equation*}

\medskip

Case 2.2: $k\geq \frac{|V(G)|}{18}$. We overcome this case by showing that there is a mapping $f:E(G)\rightarrow E(S_{10})$ with
\[|\overline{V(f)}|\leq \frac{|V(G)|-6\cdot k}{4},\]
and the mapping we get satisfies for every $v\in \overline{V(f)}$ there are $e,e\in \partial_G(v)$ with $f(e)=f(e')$.
If this is true then
\[|\overline{V(f)}|\leq \frac{|V(G)|-6\cdot k}{4}\leq \frac{|V(G)|-6\cdot \frac{|V(G)|}{18}}{4}=\frac{|V(G)|}{6}.\]
Therefore,
\begin{equation*}
|V(f)|\geq \frac{5}{6}\cdot |V(G)|.
\end{equation*}

Thus, it suffices to show that there is a mapping $f:E(G)\rightarrow E(S_{10})$ with
\[|\overline{V(f)}|\leq \frac{|V(G)|-6\cdot k}{4},\]
and the mapping we get satisfies: for every $v\in \overline{V(f)}$ there are $e,e\in \partial_G(v)$ with $f(e)=f(e')$. We will prove something stronger: there is a mapping $f:E(G)\rightarrow E(S_{10})$ with
\[|\overline{V(f)}|\leq \frac{|V(G)|-6\cdot c}{4},\]
and the mapping we get satisfies: for every $v\in \overline{V(f)}$ there are $e,e\in \partial_G(v)$ with $f(e)=f(e')$. Here $c=t+k$ is the number of bridges of $G$. Since $k\leq c$, this will imply the desired inequality.

If $c=0$, then $G$ is bridgeless, hence by Theorem \ref{thm:KMZJCTB} there is a coloring $f$ with $\overline{V(f)}=\emptyset$. Thus, our statement is trivial. Assume that all bridges of $G$ are trivial. Let $c=t$ be their number. We consider two cases:

\medskip

Case 2.2.A: $t\geq \frac{|V|}{10}$. Then:

\[|V(G)|=|V(G')|+3\cdot t\geq |V(G')|+3\cdot \frac{|V(G)|}{10},\]
or
\[|V(G')|\leq \frac{7|V(G)|}{10},\]
or as we overcame Case 1:
\[|\overline{V(f)}|\leq \frac{|V(G')|}{5}\leq \frac{7|V(G)|}{50}<\frac{|V(G)|}{7}.\]
Therefore,
\begin{equation*}
|V(f)|> \frac{6}{7}\cdot |V(G)|> \frac{5}{6}\cdot |V(G)|,
\end{equation*}
or,
\begin{equation*}
|V(f)|> \frac{5}{6}\cdot |V(G)|.
\end{equation*}
Moreover, as before for every $v\in \overline{V(f)}$ there are $e,e\in \partial_G(v)$ with $f(e)=f(e')$.

\medskip

Case 2.2.B: $t\leq \frac{|V|}{10}$. Then, by Corollary \ref{cor:kBridgeColoring},
\[|\overline{V(f)}|\leq c=t\leq \frac{|V(G)|-6\cdot t}{4},\]
which is our inequality, and note that as before for every $v\in \overline{V(f)}$ there are $e,e\in \partial_G(v)$ with $f(e)=f(e')$. Thus, we can assume that we have at least one non-trivial bridge, that is, $k\geq 1$.



Now, consider a graph $G$ with a non-trivial bridge $e=uv$. Consider two cubic graphs $G_1$ and $G_2$ obtained from $G-e$ by attaching a copy of $S_4$ to $u$ and $v$, respectively. Note that
\[|V(G_1)|+|V(G_2)|=|V(G)|+6,\]
and
\[c_1+c_2=c+1.\]
Here $c_1$ and $c_2$ denote the number of bridges of $G_1$ and $G_2$, respectively. Since $e$ is a non-trivial bridge, we have that $|V(G_1)|, |V(G_2)|<n$. Thus, by induction, we have that there are corresponding mappings $f_1$ and $f_2$ of $G_1$ and $G_2$, respectively. Note that since $e=uv$ is a bridge, we have that $u\in \overline{V(f_1)}$, and $v\in \overline{V(f_2)}$. Since roots of our bridges belong to $\overline{V(f_1)}$ and $ \overline{V(f_2)}$, we have
\[|\overline{V(f)}|\leq |\overline{V(f_1)}|+|\overline{V(f_2)}|. \]
Thus,
\begin{align*}
    |\overline{V(f)}| &\leq |\overline{V(f_1)}|+|\overline{V(f_2)}|\leq \frac{|V(G_1)|-6\cdot c_1}{4}+\frac{|V(G_2)|-6\cdot c_2}{4}\\
    &=\frac{|V(G_1)|+|V(G_2)|-6\cdot (c_1+c_2)}{4}=\frac{|V(G)|+6-6(c+1)}{6}=\frac{|V(G)|-6\cdot c}{4}.
\end{align*}

    \end{proof}

\section{Conclusion and future work}

In this paper, we considered the problem of finding a lower bound for $|V(f)|$ when one tries to color edges of every cubic graph with edges of $S_{10}$, or edges of every cubic graph with a perfect matching with edges of $S_{12}$. We provided the first lower bound for the second case, and we improved the bound for $S_{10}$ from $\frac{4}{5}$ (obtained in \cite{AuJC2018}) to $\frac{5}{6}$. The second improvement heavily relies on Theorem \ref{thm:KMZJCTB} that is proved in \cite{KMZ22}, recently. 

There are some questions that worth further investigation. First, the proof of our $\frac{4}{5}$-bound does not rely on Theorem \ref{thm:KMZJCTB}, hence it is quite plausible that one can improve it by trying to use this strong result. Next, it would be very interesting to improve our present $\frac{5}{6}$ bound.

Since there are cubic graphs with perfect matchings that do not admit $S_{10}$-colorings and $S_{12}$-colorings (see \cite{WolfAMC2026}), it would be interesting to prove best possible lower bounds for $|V(f)|$ in these two classes.

Let $P_{12}$ be the bridgeless cubic graph obtained from $P_{10}$ by replacing one of its vertices with a triangle. The paper \cite{HM2019} predicts that if a connected cubic graph $H$ does not admit a $P_{12}$-coloring, then the problem of deciding whether a given cubic graph admits an $H$-coloring is polynomial-time solvable. It is highly desirable to verify this conjecture when $H=S_6$. Since all bridgeless cubic graphs admit an $S_6$-coloring, it may be possible to use Theorem \ref{thm:KMZJCTB} to show that checking $S_6$-colorability is solvable by a polynomial time algorithm. As \cite{KMZ22} argues, $S_6$-colorability is equivalent to checking whether a given cubic graph admits a pair of perfect matchings $F_1$ and $F_2$, such that $G-F_1-F_2$ is bipartite.

\section*{Acknowledgement} The second author would like to thank Hrant Khachatryan for suggesting him to work with trivial/non-trivial bridges in cubic graphs.


\section*{References}
\bibliographystyle{elsarticle-num}

\end{document}